\documentclass[12pt, a4 paper,reqno]{amsproc}
\usepackage{epsfig,enumerate,amsmath,amsfonts,amssymb,amsthm, mathrsfs}
\theoremstyle{definition}
\newtheorem{defn}{Definition}
\newtheorem{theorem}{Theorem}
\newtheorem{lemma}{Lemma}
\newtheorem{example}[theorem]{Example}
\DeclareMathOperator{\RE}{Re}

\newcommand{\te}{\theta}

\thanks {The research work of the last  author is supported by research fellowship from University Grants Commission (UGC), New Delhi.}

\date{no date}
\title[On Solution of]{On Solution of Second Order Complex Differential Equation}
\begin{document}
\author[D. Kumar, S. Kumar and M. Saini]{Dinesh Kumar, Sanjay Kumar and Manisha Saini} 
\address{Department of Mathematics, Deen Dayal Upadhyaya College, University of Delhi,
New Delhi--110 078, India }
\email{dinukumar680@gmail.com}

\address{Department of Mathematics, Deen Dayal Upadhyaya College, University of Delhi,
New Delhi--110 078, India }

\email{sanjpant@gmail.com}
\address{Department of Mathematics, University of Delhi,
Delhi--110 007, India}

\email{sainimanisha210@gmail.com }

\begin{abstract}
In this paper, we establish transcendental entire function $A(z)$ and polynomial $B(z)$ such that the differential equation $f''+A(z)f'+B(z)f=0$, has all non-trivial solution of infinite order. We use the notion of \emph{critical rays} of the function $e^{P(z)}$, where $A(z)=d(z)e^{P(z)}$ with some restrictions.
\end{abstract} 

\keywords{entire function, meromorphic function, order of growth, complex differential equation}

\subjclass[2010]{34M10, 30D35}

\maketitle

\section{Introduction}
Consider a second order linear differential equation of the form
\begin{equation}\label{sde}
f''+A(z)f'+B(z)f=0; \qquad \qquad  B(z) \not \equiv 0 
\end{equation}
where $A(z)$ and $B(z) $ are entire functions. It is well known that all solutions of the equation (\ref{sde}) are entire functions \cite{helle}, \cite{lainebook}. For an entire function $f,$ we define the order and lower order of $f$ as follows
$$ \rho(f) = \limsup_{r \rightarrow \infty} \frac{\log^{+} \log^{+} M(r, f)}{\log r} , \qquad \mu(f)=\liminf_{r \rightarrow \infty} \frac{\log^+ \log^+ M(r, f)}{\log r}$$
where $ M(r,f)=\max_{\mid z \mid=r}{\mid f(z) \mid}$ is maximum modulus of $f$. Using Wiman-Valiron theory, it is proved that the equation (\ref{sde}) has all solutions of finite order if and only if $A(z)$ and $B(z)$ are polynomials \cite{lainebook}. Therefore, if either $A(z)$ or $B(z)$ is transcendental entire function then almost all solutions of equation (\ref{sde}) are of infinite order. Therefore, it is natural to find conditions on $A(z)$ and $B(z)$ such that all solutions of the equation (\ref{sde}) are of infinite order.  Gundersen \cite{finitegg} proved the following result:

{\bf {Theorem A.}} 
A necessary condition for equation (\ref{sde}) to have a non-trivial solution $f$ of finite order is 
\begin{equation}\label{necc}
\rho(B)\leq \rho(A).
\end{equation} 
\\
We illustrate it with some examples,

\begin{example} $f(z)=e^{z}$ satisfies $f''+e^{z}f'-(e^{z}+1)f=0,$ where $\rho(A)=\rho(B)=1$.
\end{example}\label{eg1} 
\begin{example}\label{eg2}
With $A(z)=e^z$ and $B(z)=-1$ equation (\ref{sde}) has finite order solution  $f(z)=1-e^{-z}$, where $\rho(B)<\rho(A).$
\end{example}

Thus if $\rho(A)<\rho(B)$ then all solutions of the equation (\ref{sde}) are of infinite order. However, given necessary condition is not sufficient, for example
\\

\begin{example}\cite{heitt}
If $A(z)=P(z)e^{z}+Q(z)e^{-z}+R(z)$, where $ P, Q, R$ are polynomials and $B(z)$ is an entire function with $\rho(B)<1$ then $\rho(f)$ is infinite; for all non-trivial solutions of the equation (\ref{sde}).
\end{example}

Frei \cite{frei} showed that the differential equation
\begin{equation}\label{freieq}
f''+e^{-z}f'+B(z)f=0
\end{equation}
with $B(z)$ a polynomial, then the equation (\ref{freieq}) has a non-trivial solution $f$ of finite order if and only if $B(z)=-n^2, n \in \mathbb{N}$. Ozawa \cite{ozawa} proved that the equation (\ref{freieq}) with $B(z)=az+b, a\neq 0$ has all solution of infinite order. Amemiya and Ozawa \cite{ame}, and Gundersen \cite{ggpol} studied the equation (\ref{freieq}) for $B(z)$ being a particular polynomial. Langley \cite{lang}, proved that this is true for any non-connstant polynomial.
 Gundersen \cite{finitegg} proved the following result:
\begin{theorem}\label{thm1}
Let $f$ be a non-trivial solution of the equation (\ref{sde}) where either 
\begin{enumerate}[(i)]
\item $\rho(B)< \rho(A)<\frac{1}{2}$\\
or
\item $A(z)$ is transcendental entire function with $\rho(A)=0$ and $B(z)$ is a polynomial
\end{enumerate}
then $\rho(f)$ is infinite.
\end{theorem}
Hellerstein, Miles and Rossi \cite{heller} proved Theorem \ref{thm1} for $\rho(B)<\rho(A)=\frac{1}{2}.$
\\

J.R. Long introduced the notion of the deficient value and Borel direction into the studies of the equation (\ref{sde}). For the definition of deficient value, Borel direction and function extremal for Yang's inequality one may refer to \cite{yang}. 
\\

 In \cite{extremal}, J.R .Long proved that if $A(z)$  is an entire function extremal for Yang's inequality and $B(z)$  a transcendental entire function with $\rho(B)\neq \rho(A)$, then all solution of the equation (\ref{sde}) are of infinite order. In \cite{jlongfab}, J.R. Long replaced the condition $\rho(B) \neq \rho(A)$ with the condition that $B(z) $ is an entire function with \emph{Fabry gaps}.
\\

 X.B.Wu  \cite{wu} proved that if $A(z)$ is a non-trivial solution of $w''+Q(z)w=0$, where $Q(z)= b_mz^m+\ldots +b_0, \quad b_m \neq 0$ and $B(z)$ be an entire function with $\mu(B)<\frac{1}{2}+ \frac{1}{2(m+1)}$, then all solutions of equation (\ref{sde}) are of infinite order. J.R.  Long \cite{jlongfab} replaced the condition $\mu(B)< \frac{1}{2}+\frac{1}{2(m+1)}$ with $B(z)$ being an entire function with \emph{ Fabry gaps} such that $\rho(B) \neq \rho(A)$.
\\
Furthermore, J.R. Long \cite{jlong} proved the following Theorem:
\\

{\bf{Theorem B.}} Let $A(z)=d(z)e^{P(z)}$, where $d(z)( \not \equiv 0)$ is an entire function and $P(z)=a_nz^n +\ldots +a_0$ is a polynomial of degree $n$ such that $\rho(d)<n$. Let $B(z)=b_mz^m +\ldots +b_0$ be a non-constant polynomial of degree $m$, then all non-trivial solutions of the equation (\ref{sde}) have infinite order if one of the following condition holds:
\begin{enumerate}
\item $m+2 <2n$;
\item $m+2>2n$ and $m+2 \neq 2kn$ for all integers $k$; 
\item $m+2=2n$ and $\frac{a_n^2}{b_m}$ is not a negative real.
\end{enumerate}
Motivated by Theorem B, we consider entire functions $A(z)$ and $B(z)$ such that $\rho(A)>n$ and $B(z)$ a polynomial. To state and prove our theorem we give some definitions and notations below:
\begin{defn}\label{def1}\cite{jlong}
Let $P(z)=a_{n}z^n+a_{n-1}z^{n-1}+\ldots +a_0$, $a_n\neq0$ and $\delta(P,\theta)=\RE(a_ne^{\iota n \theta})$. 
A ray $\gamma = re^{\iota \theta}$  is called \emph{critical ray} of $e^{P(z)}$ if $\delta(P,\theta)=0.$ 
\end{defn}
It can be easily seen  that there are $2n$  different critical rays of $e^{P(z)}$ which divides the whole complex plane into $2n$ distinict sectors of equal length $\frac{\pi}{n}.$  Also $\delta(P,\theta)>0$ in $n$ sectors and $\delta(P,\theta)<0$ in remaining $n$ sectors. We note that $\delta(P,\theta)$ is alternatively positive and negative in the $2n$ sectors.\\
We now fix some notations, \\

 $ E^+ = \{ \theta \in [0,2\pi]: \delta(P,\theta)\geq 0\}$ and $E^- = \{ \theta \in [0,2\pi]: \delta(P,\theta)\leq 0 \}.$
\\
Let $\alpha>0$ and $\beta>0$ be such that $0\leq \alpha<\beta\leq 2\pi$ then 
\[S(\alpha,\beta)= \{z\in \mathbb{C}: \alpha<\arg z <\beta \}\]
We now recall the notion of an entire function to blow up and decay to zero exponentially \cite{wu}.
\begin{defn} \label{blowsup}
Let $A(z)$ be an entire function with order $\rho(A)\in (0,\infty)$. Then $A(z)$ \emph{blows up exponentially} in $\bar{S}(\alpha, \beta)$ if for any  $\theta\in S(\alpha, \beta)$ we get,
$$\lim_{r \rightarrow \infty} \frac{\log \log \mid A(re^{\iota \theta})\mid}{\log r} =\rho(A).$$

We say $A(z)$ \emph{decays to zero exponentially} in $\bar {S}(\alpha, \beta)$ if for any $\te \in S(\alpha, \beta)$ 
$$ \lim_{r \rightarrow \infty} \frac{\log \log \left | \frac{1}{A(r e^{\iota \te})}\right |}{\log r}=\rho(A)$$
\end{defn}

We illustrate these notions with an example

\begin{example}
The function $f(z)=e^z $  has two \emph{critical rays} namely $-\frac{\pi}{2}$ and $\frac{\pi}{2}$. It is easy to show that $f(z)$ \emph{ blows up exponentially} in $\bar{S}(-\frac{\pi}{2},\frac{\pi}{2})$ and  \emph{decays to zero exponentially} in $\bar{S}(\frac{\pi}{2}, \frac{3\pi}{2})$.

\end{example}
We are now able to  state our main theorem:
\begin{theorem}\label{main theorem}
Consider a transcendental entire function $A(z)=d(z)e^{P(z)}$,  where $P(z)$ is a non-constant polynomial of degree $n$ and $\rho(d)>n$. Assume that $d(z)$ is bounded away from zero and exponentially blows up in $E^+$ and $E^-$ respectively and let $B(z)$ be a polynomial. Then all non-trivial solutions of the equation (\ref{sde}) are of infinite order.
\end{theorem}

The paper is organised as follows: in Section 2, we have stated preliminary lemmas and proved some required results. In Section 3, we have proved Theorem \ref{main theorem}. 
\\

\section{Auxiliary Result}
In this section, we present some known results. Next two lemmas are due to Gundersen which has been used extensively over the years.
\\

\begin{lemma}\cite{log gg} \label{ggllog}
Let $f$ be a trancendental meromorphic function with finite order and $(k,j)$ be a pair of integers that satisfies $k>j\geq0$. Then for $\epsilon>0$ there exists a set $E \subset [0,2\pi]$ with linear measure zero such that for $\theta \in [0,2\pi) \setminus E$ there exist $R(\theta)>1$ such that 
\begin{equation}\label{gglog}
\left| \frac{f^{(k)}(z)}{f^{(j)}(z)}\right | \leq \mid z \mid^{(k-j)(\rho(f)-1+\epsilon)}
\end{equation} for$\mid z \mid >R(\theta)$ and $\arg z =\theta.$

\end{lemma}

\begin{lemma}\cite{finitegg}
Let $f$ be analytic on a ray $\gamma=re^{\iota \theta}$ and suppose that for some constant $\alpha>1$ we have 
$$ \left | \frac{f'(z)}{f(z)}\right |=O(\mid z\mid^{-\alpha})$$ 
as $z\rightarrow \infty$ along $\arg z =\theta.$ Then there exists a constant $c\neq0$ such that $f(z)\rightarrow c$ as $z\rightarrow \infty $ along $\arg z =\theta.$
\end{lemma}\label{ggfinite}

We now prove a result which would be required for proving Theorem \ref{main theorem}.

\begin{lemma}\label{my lemma}
Let $A(z)=d(z)e ^{P(z)}$ be an entire function, where $P(z)$ is a polynomial of degree $n$ and $d(z)$ satisfies the condition of Theorem \ref{main theorem}.  Then there exists a set $E \subset [0,2\pi]$ of linear measure zero such that 
 for $\epsilon >0$ the following holds:
\begin{enumerate}[(i)]
\item for $\theta \in E^+ \setminus E$ there exists $R(\theta)>1$ such that \\
$$\mid A(re^{\iota \theta})\mid \geq \exp \big((1-\epsilon) \delta(P,\theta) r^n \big)$$
for $r>R(\theta)$
\item for  $\theta\in E^- \setminus E$ there exists $R(\theta)>1$ such that 
$$\mid A(re^{\iota \theta})\mid \geq \exp \big((1-\epsilon) \delta(P,\theta)r^n \big)$$
for $r>R(\theta).$

\end{enumerate}
\end{lemma}
\begin{proof}
Here $A(z)=d(z)e^{P(z)}=h(z)e^{a_nz^n}$, where $h(z)$ is also an entire function. Let $E=\{\theta\in[0,2\pi]: \delta(P,\theta)=0\}$. This means that $E$ is set of critical rays of $e^{P(z)}$, which implies that $E$ has linear measure zero. Let $\epsilon>0$. Then
\begin{enumerate}[(i)]
 \item for $\delta(P,\theta)>0$,  $\exp(-\epsilon \delta(P,\theta)r^n)\rightarrow 0$ as $r \rightarrow \infty$. Thus there exists $R(\theta)>1$ such that 

\begin{equation}\label{eqh}
\exp(-\epsilon \delta(P,\theta)r^n) \leq \mid h(re^{\iota \theta}) \mid 
\end{equation}
for $r>R(\theta)$. Now

\begin{equation}\label{eqA}
 \mid \exp(a_n(re^{\iota \theta})^n)\mid=\exp(\delta(P,\theta)r^n)
\end{equation}
 Thus using (\ref{eqh}) and (\ref{eqA}) we have
$$ \mid A(re^{\iota \theta}) \mid =\mid h(re^{\iota \theta})\mid \mid \exp(a_n (re^{\iota \theta})^n)\mid \geq \exp \big((1-\epsilon) \delta(P,\theta)r^n \big) $$
for $\theta \in E^+ \setminus E $ and $ r>R(\theta).$ 
\item Since $d(z)$ blows up exponentially in $E^-$ therefore, $h(z)$ also blows up exponentially in $E^-$. Let $\epsilon >0$ and $\delta(P,\theta)<0$. Then $\rho(\exp(-\epsilon a_n z^n))=n<\rho(h)$. Thus, using definition \ref{blowsup}, for any $\theta\in E^ - \setminus E$  there exists $R(\theta)>1$ such that  
\begin{equation}\label{equh}
\mid h(re^{\iota \theta})\mid\geq \exp(-\epsilon \delta(P,\theta)r^n)
\end{equation}
for $r>R(\theta)$. Using equation (\ref{eqA}) and (\ref{equh}) we obtain that
$$\mid A(re^{\iota \theta})\mid \geq \exp \big((1-\epsilon)\delta(P,\theta)r^n \big)$$
for $r>R(\theta)$ and $\theta\in E^- \setminus E.$

\end{enumerate}
\end{proof}

\section{Proof of Main Theorem}
In this section we will establish Theorem \ref{main theorem} which is the main result of this paper. 
\begin{proof}If $\rho(A) = \infty$ then it is obvious that $\rho(f)=\infty$, for all non-trivial solution $f$ of the equation (\ref{sde}).Therefore, let us suppose that $\rho(A)<\infty$ and there exists a non-trivial solution $f$ of the equation (\ref{sde}) such that $\rho(f)<\infty$. Then from Lemma \ref{ggllog}, we have that there exist $E_1 \subset [0,2\pi]$ of linear measure zero and  $m>0$ such that,
\begin{equation}\label{eq1proof}
\left | \frac{f''(re^{\iota\theta})}{f(re^{\iota \theta})} \right | \leq r^{m}
\end{equation}
for $\theta \in [0,2\pi] \setminus E_1$ and $r>R(\theta)$.  From Lemma \ref{my lemma}, part (i)  we have, 
\begin{equation}\label{eq2proof}
 \mid A(re^{\iota \theta}) \mid \geq \exp \left(\frac{1}{2} \delta(P,\theta) r^n \right )
\end{equation} 
for $\theta \in E^+ \setminus E_2$ and $r> R'(\theta)$ where $E_2 $ is set of critical rays of $e^{P(z)}$ of linear measure zero. 
Using equation (\ref{sde}), (\ref{eq1proof}) and (\ref{eq2proof}), for $\theta \in E^+ \setminus \big( E_1\cup E_2 \big)$ we get,
$$ \left | \frac{f'(re^{\iota \theta})}{f(re^{\iota \theta})}\right | \rightarrow 0$$
as $r\rightarrow \infty$. This implies that for  $\theta \in E^+ \setminus \big( E_1  \cup E_2 \big)$ 
\begin{equation}\label{eq3proof}
\left| \frac{f'(re^{\iota\theta})}{f(re^{\iota \theta})} \right | = O\left (\frac{1}{r^2} \right )
\end{equation}
  as $r \rightarrow \infty$. From Lemma \ref{ggfinite},
 \begin{equation}\label{eq}
f(re^{\iota \theta}) \rightarrow a
\end{equation}
as $r \rightarrow \infty$,  for $\theta \in E^+ \setminus \big( E_1 \cup E_2 \big) $,  where $a$ is a non-zero finite constant. Applying Maximum Modulus principle for the function $f$ over the domain $E^+$  this can be concluded that $f$ is bounded over $E^+$. Now using Phragm$\acute{e}$n- Lindel$\ddot{o}$f principle, 
\begin{equation}\label{eq5proof}
f (re^{\iota \theta})\rightarrow a
\end{equation}
as $r \rightarrow \infty$, for $\theta \in E^+$. 
\\
Lemma \ref{my lemma}, part (ii) implies that, 
\begin{equation}\label{eq4proof}
 \mid A(re^{\iota \theta}) \mid \geq \exp \left (-\frac{1}{2} \delta(P,\theta) r^n \right)
\end{equation} 
for $\theta \in E^- \setminus E_2$ and for $r>R''(\theta)$. Using equation (\ref{sde}), (\ref{eq1proof}) and (\ref{eq4proof}) we have,
\begin{equation}\label{eq6proof}
\left| \frac{f'(re^{\iota\theta})}{f(re^{\iota\theta})} \right| \rightarrow 0
\end{equation}  
as $r\rightarrow \infty$, for $\theta \in E^- \setminus \big(E_1 \cup E_2\big)$. From here we can obtain equations (\ref{eq3proof}) for $\theta \in E^- \setminus \big( E_1 \cup E_2 \big)$. Again using Maximum Modulus principle for function $f$ over the domain $E^-$  we get $f$  is bounded over $E^-$. Which further using Phragm$\acute{e}$n-Lindel$\ddot{o}$f principle implies,
\begin{equation}\label{eq7proof}
 f(re^{\iota \theta}) \rightarrow b
\end{equation} 
as $r \rightarrow \infty$, for $\theta \in E^-$, where $b$ is a non-zero finite constant.
Again using Phragm$\acute{e}$n-Lindel$\ddot{o}$f principle we get, 
\begin{equation}
 f(re^{\iota \theta}) \rightarrow a
\end{equation}  
as $r\rightarrow \infty$, for all $\theta \in E^+ \cup E^- $, which is a contradiction to the Liouville's theorem.
\end{proof}
{\bf{Acknowlegement:}} We are thankful to Professor Gundersen for reading the paper and suggesting many things. In fact, he has asked to construct an example in support of the main theorem. We are still working on construction of such type of example. We invite readers for their comments. Paper is not for publication in any academic journal as of now.

\end{document}